\documentclass{amsart}

\usepackage{graphicx}
\usepackage{amssymb}
\usepackage{mathtools}
\usepackage{pdfsync}
\usepackage{multido}

\usepackage[T1]{fontenc}
\usepackage{hyperref} 

\hypersetup{
 colorlinks=true,       
    linkcolor=cyan,          
    citecolor=green,        
    filecolor=magenta,      
   urlcolor=cyan           
   }

\numberwithin{equation}{section}
\numberwithin{figure}{section}

\allowdisplaybreaks
\DeclareSymbolFont{bbold}{U}{bbold}{m}{n}
\DeclareSymbolFontAlphabet{\mathbbold}{bbold}
\newcommand{\ind}{\mathbbold{1}}

\theoremstyle{plain} \newtheorem{theorem}{Theorem}[section]
\theoremstyle{plain} 
\theoremstyle{plain} 
\theoremstyle{plain} \newtheorem{corollary}[theorem]{Corollary}
\theoremstyle{definition} \newtheorem{definition}[theorem]{Definition}
\theoremstyle{definition} 
\theoremstyle{remark} \newtheorem{remark}[theorem]{Remark}
\theoremstyle{remark}

\newcommand{\bN}{\mathbb{N}}
\newcommand{\bP}{\mathbb{P}}

\newcommand{\bR}{\mathbb{R}}
\newcommand{\bS}{\mathbb{S}}

\def\rr{\mathbf{r}}
\def\tt{\mathbf{t}}
\def\aa{\mathbf{a}}
\def\bb{\mathbf{b}}
\def\cc{\mathbf{c}}
\def\dd{\mathbf{d}}

\renewcommand{\ss}{\mathbf{s}}

\def\LL{\mathbf{L}}
\def\RR{\mathbf{R}}

\def\TT{\mathbf{T}}

\begin{document}

\title[Radix sort chain]{Radix sort trees in the large}

\author{Steven N. Evans}
\address{Department of Statistics\\
         University  of California\\ 
         367 Evans Hall \#3860\\
         Berkeley, CA 94720-3860 \\
         U.S.A.}

\email{evans@stat.berkeley.edu}
\thanks{SNE supported in part by NSF grant DMS-0907630, NSF grant DMS-1512933,
and NIH grant 1R01GM109454-01. 
AW supported in part by DFG priority program 1590.}

\author{Anton Wakolbinger}
\address{Institut f\"ur Mathematik \\
         Goethe-Universit\"at \\
         60054 Frankfurt am Main\\
         Germany}

\email{wakolbinger@math.uni-frankfurt.de}



\date{\today}

\begin{abstract} The trie-based radix sort algorithm stores pairwise different infinite binary strings 
in the leaves of a binary tree in a way that the Ulam-Harris coding of each leaf equals 
a prefix (that is, an initial segment) of the corresponding string, with the prefixes being of minimal length 
so that they are pairwise different. We investigate the {\em radix sort tree chains} 
-- the tree-valued Markov chains that arise when successively storing infinite binary strings 
$Z_1,\ldots, Z_n$, $n=1,2,\ldots$ according to the trie-based radix sort algorithm, 
where the source strings $Z_1, Z_2,\ldots$ are independent and identically distributed. 
We  establish a bijective correspondence between  the full Doob--Martin boundary of the radix sort tree chain with a
{\em symmetric Bernoulli source} (that is, each $Z_k$ is a fair coin-tossing sequence)
and the family of radix sort tree chains for which the common distribution of the $Z_k$ 
is a diffuse probability measure on $\{0,1\}^\infty$. 
In essence, our result characterizes all the  ways that it is possible
 to condition such a chain of radix sort trees consistently on its behavior ``in the large''.
\end{abstract}

\maketitle

\tableofcontents

\section{Introduction}
\label{S:intro}

Various sorting algorithms proceed by storing the data in the leaves of a tree. If the data are infinite {\em binary strings} $z_1,\ldots, z_n \in \{0,1\}^\infty$, then a natural choice for the tree is the rooted binary tree with $n$ leaves chosen such that the Ulam-Harris coding of each of the leaves coincides with a finite initial segment (otherwise called a prefix or left factor) of one of the $z_j$, and such that these initial segments are pairwise different and have minimal length (see below for a fuller description). This data structure is the basis of the {\em Radix Sort} algorithm. The tree $R(z_1,\ldots, z_n)$ in whose leaves the $n$ strings are stored is sometimes called a {\em trie}, alluding to the word re{\em trie}val.

When the $n$ strings are random, drawn i.i.d. from a diffuse probability distribution $\nu$ on   $\{0,1\}^\infty$, then this construction gives rise to a random tree $\prescript{\nu}{}{R}_n:= R(Z_1,\ldots, Z_n)$. In order to obtain a probabilistic analysis of the Radix Sort algorithm, asymptotic properties of these random trees as $n\to \infty$ have been considered for the {\em symmetric Bernoulli} or {\em unbiased memoryless source model}, where $\nu$ is the fair coin tossing measure, e.g. in 
\cite{MR1140708} ch. 5 and \cite {MR3077154} \textsection 5.2.2., and for more general inputs of random strings in \cite{MR1816272}. The  {\em density model}, where $\nu$ is the image under the binary expansion of an absolutely continuous probability measure on $[0,1]$, was considered in \cite{MR1161060}. {\em Dynamical sources} appear in \cite{MR1887308}; these include {\em Markovian inputs}, where $\nu$ is the shift-invariant distribution of a  Markov chain, see \cite{JaSp}, \cite{LNS}.

In this paper we analyze the tree-valued Markov chains $(\prescript{\nu}{}{R}_n)_{n\in \mathbb N}$  from a more synoptic point of view. We show that any such chain is a harmonic transform of the Markov chain $(\prescript{\gamma}{}{R}_n)_{n\in \mathbb N}$, with $\gamma$ the fair coin-tossing measure, and we  prove that the family $(\prescript{\nu}{}{R}_n)_{n\in \mathbb N}$ as $\nu$ varies constitute the full Doob--Martin boundary of  $(\prescript{\gamma}{}{R}_n)_{n\in \mathbb N}$. Loosely speaking, this means that all consistent ways of conditioning a chain of radix sort trees  ``in the large'' are described by precisely the family  $(\prescript{\nu}{}{R}_n)_{n\in \mathbb N}$.

In order to state our main result more formally, we first fix some notation.  Denote by $\{0,1\}^\star:=\bigsqcup_{k=0}^\infty\{0,1\}^k$ the set of finite tuples
or {\em words} drawn from the alphabet $\{0,1\}$ (with the empty word $\emptyset$
allowed) -- the symbol $\bigsqcup$ emphasizes that this is a disjoint union.  
Write an $\ell$-tuple $v=(v_1, \ldots, v_\ell) \in \{0,1\}^\star$ more simply as
$v_1 \ldots v_\ell$ and set $|v| = \ell$.  Define a directed graph with vertex set
$\{0,1\}^\star$ by declaring that if
$u= u_1 \ldots u_k$ and $v=v_1 \ldots v_\ell$ are two words,
then $(u,v)$ is a directed edge (that is, $u \rightarrow v$)
if and only if $\ell=k+1$ and $u_i=v_i$
for $i=1,\ldots,k$.  Call this directed graph
the {\em complete rooted binary tree}.  Say that $u < v$ for
two words $u= u_1 \ldots u_k$ and $v=v_1 \ldots v_\ell$ in $\{0,1\}^\star$ if
$k < \ell$ and $u_1 \ldots u_k = v_1 \ldots v_k$;
that is, $u < v$ if there exist words 
$w_0, w_1, \ldots, w_{\ell-k}$ with
$u = w_0 \to w_1 \to \ldots \to w_{\ell-k} = v$.
This partial order extends to $\{0,1\}^\star \sqcup \{0,1\}^\infty$ in the obvious
way: if $u \in \{0,1\}^\star$ and $v \in \{0,1\}^\infty$, then $u < v$
when $u = u_1 \ldots u_k$ and $v=v_1 v_2 \ldots$ with $u_1 \ldots u_k = v_1 \ldots v_k$
(and no two elements of $\{0,1\}^\infty$ are comparable).
It will be convenient to introduce the notation  
$\tau(y) := \{z \in \{0,1\}^\infty : y < z\}$ 
for $y \in \{0,1\}^\star$. 

A {\em finite rooted binary tree}
is a non-empty subset $\tt$ of 
$\{0,1\}^\star$ with the property that if $v \in \tt$
and $u \in \{0,1\}^\star$ is such that $u \rightarrow v$, then $u \in \tt$.
The vertex $\emptyset$ (that is, the empty word)
belongs to any such tree $\tt$ and
is the {\em root} of $\tt$. The {\em leaves} of $\tt$ are the elements $v \in \tt$
such that if $v \to w$, then $w \notin \tt$, and we use the notation
$\LL(\tt)$ for the leaves of $\tt$. A finite rooted binary tree is uniquely determined by its
leaves: it is the smallest rooted binary tree that contains the set of leaves and it
consists of the leaves and the points $u \in \{0,1\}^\star$ such that 
$u < v$ for some leaf $v$. In general, write
\[
\TT(y_1, \ldots, y_m) := \bigcup_{j=1}^m \{u \in \{0,1\}^\star : u \le y_j\}
\]
for the smallest finite rooted binary tree
containing $y_1, \ldots, y_m \in \{0,1\}^\star$; the leaves of this tree form a subset
of $\{y_1, \ldots, y_m\}$ and this subset is proper if and only if $y_i < y_j$ for some
pair $1 \le i \ne j \le m$.

A collection $z_1, \ldots, z_n$ of distinct elements of $\{0,1\}^\infty$ determines
a finite rooted binary tree in the following manner. 
For $n = 1$, put $H_{1,1}(z_1) := 0$ and $\zeta_{1,1}(z_1) := \emptyset$.  
For $n \ge 2$ and $1 \le j \le n$, let 
\[
H_{n,j}(z_1, \ldots, z_n) := \min\{\ell : (z_{j,1}, \ldots, z_{j,\ell}) \ne (z_{k,1}, \ldots, z_{k,\ell}), \; k \ne j\}
\]
be the minimal length at which a prefix of $z_j$ differs from the
prefixes of the same length of all the other $z_k$, $k \neq j$,  
and denote the corresponding prefix by 
\begin{align}\label{zeta}
\zeta_{n,j}(z_1, \ldots, z_n) := (z_{j,1}, \ldots, z_{j,H_{n,j}(z_1, \ldots, z_n)}) \in \{0,1\}^\star, \quad 1 \le j \le n.
\end{align}

The words $\zeta_{n,j}(z_1, \ldots, z_n)$, $1 \le j \le n$, are distinct
and $\zeta_{n,j}(z_1, \ldots, z_n) < z_j$ for $1 \le j  \le n$.
Note that if $\sigma$ is a permutation of $[n] := \{1,\ldots,n\}$, then
\begin{equation}
\label{symmetric_tree_build}
\zeta_{n,\sigma(j)}(z_{\sigma(1)}, \ldots, z_{\sigma(n)}) = \zeta_{n,j}(z_1, \ldots, z_n).
\end{equation}

The {\em radix sort tree determined by the input} $z_1, \ldots, z_n$ is defined as
\[
\RR(z_1, \ldots, z_n) := \TT(\zeta_{n,1}(z_1, \ldots, z_n), \ldots, \zeta_{n,n}(z_1, \ldots, z_n)).
\]
Thus, $\RR(z_1, \ldots, z_n)$ is the finite rooted binary tree whose $n$ leaves are coded by the $n$
finite strings of \eqref{zeta}.
Observe that 
\begin{equation}
\label{permutation_invariance}
\RR(z_1, \ldots, z_n) = \RR(z_{\sigma(1)}, \ldots, z_{\sigma(n)})
\end{equation}
for any permutation $\sigma$ of $[n]$.

Let $Z_1, Z_2, \ldots$ be i.i.d. $\{0,1\}^\infty$-valued random variables 
with common distribution
some diffuse probability measure $\nu$.  Then $Z_1, Z_2,\ldots$ are a.s. pairwise distinct, and on this event we
set $\prescript{\nu}{}{R}_n := \RR(Z_1, \ldots, Z_n)$.   
When $\nu$ is fair coin-tossing measure $\gamma$ (that is, $\gamma$ is the infinite product
of the uniform measure on $\{0,1\}$), we drop the~$\nu$ and simply write $R_n$ 
for $\prescript{\gamma}{}{R}_n$.  It is not hard to see that
$(\prescript{\nu}{}{R}_n)_{n \in \bN}$ is a Markov chain; we
call it a {\em radix sort tree chain}.  

Note for $y \in \{0,1\}^*$ and $n\ge k \ge 2$ that with probability one
\[
\#\{1 \le j \le n : y \le \zeta_{n,j}(Z_1, \ldots, Z_n)\} = k
\]
if and only if
\[
\#\{1 \le j \le n : Z_j \in \tau(y)\} = k,
\]
Thus,
\[
\nu(\tau(y)) 
= 
\lim_{n \to \infty} \frac{1}{n} \#\{1 \le j \le n : y \le \zeta_{n,j}(Z_1, \ldots, Z_n)\} \quad \bP-\text{a.s.}
\]
and $\nu$ can be recovered almost surely from the tail
$\sigma$-field of $(\prescript{\nu}{}{R}_n)_{n \in \bN}$;
in particular, different choices of $\nu$ result in different
distributions for $(\prescript{\nu}{}{R}_n)_{n \in \bN}$.
It follows from \eqref{permutation_invariance} and the Hewitt--Savage zero--one law that
the tail $\sigma$-field of $(\prescript{\nu}{}{R}_n)_{n \in \bN}$ is $\bP$-a.s. trivial.

In order to describe our results, we need to use some notions and facts
from Doob--Martin boundary theory.  A quick summary tailored to the sort of
setting we are in of a process  which ``goes off to infinity''
and never revisits states may be found in \cite{MR2869248, Remy}, where there are also
references to expositions of the general theory for arbitrary transient
Markov chains following on from the seminal paper \cite{MR0107098}.
Analyses of binary-search-tree and digitial-search-tree
chains from the Doob--Martin point of view are presented in \cite{MR2869248}.

Let $\bS_n$ be the set of trees that can arise as $\RR(z_1, \ldots, z_n)$
for some choice of $z_1, \ldots, z_n$ and set $\bS = \bigsqcup_{n \in \bN} \bS_n$. Of course, $\bS_1 = \{\emptyset\}$.
For $n \ge 2$, a finite rooted binary tree $\tt$ with $n$ leaves
belongs to $\bS_n$ if and only if whenever $u_1 u_2 \ldots u_{m-1} u_m \in \LL(\tt)$,
then $u_1 u_2 \ldots u_{m-1} \bar u_m \in \tt$, where $\bar 0 := 1$
and $\bar 1 := 0$.

Given a binary tree $\tt \in \bS$ with $M(\tt)$ leaves (that is, $\tt \in \bS_{M(\tt)}$), write
$ R_1^\tt, R_2^\tt, \ldots, R_{M(\tt)}^{\tt}$ for the {\em bridge} process
obtained by conditioning $R_1, \ldots, R_{M(\tt)}$
on the event $\{R_{M(\tt)} = \tt\}$.  This Markov chain has the same
backward transition probabilities as $(R_n)_{n \in \bN}$; that is,
\[
\bP\{{R}_n^\tt = \rr \, | \, {R}_{n+1}^\tt = \ss\}
=
\bP\{{R}_n = \rr \, | \, {R}_{n+1} = \ss\}
\]
for $n+1 \le M(\tt)$.

An {\em infinite bridge} for $(R_n)_{n \in \bN}$
is a Markov chain $(R_n^\infty)_{n \in \bN}$ with $R_n^\infty \in \bS_n$
for $n \in \bN$
and the same backward transition
probabilities as $(R_n)_{n \in \bN}$.  We show in Sec.~\ref{S:harmonic_examples} that
each chain $(\prescript{\nu}{}{R}_n)_{n \in \bN}$ is an infinite bridge for
$(R_n)_{n \in \bN}$.  Any infinite bridge is a
Doob $h$-transform of $(R_n)_{n \in \bN}$; that is, it has forward transition
probabilities of the form
\[
\bP\{R_{n+1}^\infty = \tt \, | \, R_n^\infty = \ss\}
=
h(\ss)^{-1} \bP\{R_{n+1} = \tt \, | \, R_n = \ss\} h(\tt),
\]
where the nonnegative function $h$ is given up to a constant multiple by
\[
h(\tt) = \frac{\bP\{R_n^\infty = \tt\}}{\bP\{R_n = \tt\}}.
\]
The function $h$ is harmonic for $(R_n)_{n \in \bN}$; that is,
\[
\sum_\tt \bP\{R_{n+1} = \tt \, | \, R_n = \ss\} h(\tt) = h(\ss).
\]
Conversely, any Markov chain with initial state the trivial
tree $\emptyset$ 
and transition probabilites that arise from those of $(R_n)_{n \in \bN}$
 through the $h$-transform construction 
for some nonnegative harmonic function $h$ 
(normalized, without loss of generality, so that $h(\emptyset) = 1$) 
is an infinite bridge.

The distribution of an infinite bridge is a mixture of distributions of
infinite bridges with almost surely trivial tail $\sigma$-fields.  Equivalently,
the collection of nonnegative harmonic functions $h$ with $h(\emptyset) = 1$ is a compact convex
set (for the product topology on $\bR_+^\bS$)
and any such function is a unique convex combination of the extreme points
of this set. In particular, there is a bijective correspondence between the extreme points of these two sets; that is between the set of infinite bridges with
trivial tail $\sigma$-fields and extremal normalized nonnegative harmonic
functions.

One way to construct infinite bridges
is to look for sequences $(\tt_k)_{k \in \bN}$ with $M(\tt_k) \to \infty$
as $k \to \infty$ such that initial segments of the finite bridges
$(R_1^{\tt_k}, R_2^{\tt_k}, \ldots, R_{M({\tt_k})}^{\tt_k})$
converge in distribution as $k \to \infty$.  
A necessary condition for an infinite bridge to have an 
almost surely trivial tail $\sigma$-field is that is arises from such a construction.

The nonnegative harmonic function corresponding to an infinite bridge constructed in this way
(normalized to have $h(\emptyset)=1$) is
\begin{equation}
\label{harmonic_limit_Doob-Martin}
h(\ss) = \lim_{k \to \infty} K(\ss, \tt_k),
\end{equation}
where
\begin{equation}
\label{def_Doob-Martin_kernel}
K(\ss,\tt)
:=
\frac{\bP\{R_{M(\ss)}^{\tt} = \ss\}}{\bP\{R_{M(\ss)} = \ss\}}
=
\frac{\bP\{R_{M(\tt)} = \tt \, | \, R_{M(\ss)} = \ss\}}
{\bP\{R_{M(\tt)} = \tt\}}
\end{equation}
is the {\em Doob--Martin kernel}.
A necessary condition for a normalized nonnegative harmonic function to be an extreme point is that it arises as such a limit.

The following is our main result characterizing all the ways that it is possible
to condition the radix sort tree chain with
inputs distributed according to fair coin-tossing measure.
We prove this result in Section~\ref{S:T:summary_proof}.

\begin{theorem}
\label{T:summary}
An infinite bridge for the radix sort tree chain with inputs distributed according
to fair coin-tossing measure on $\{0,1\}^\infty$ has an almost surely trivial
tail $\sigma$-field if and only if it is a Markov chain with the same distribution as
the radix sort tree chain with inputs distributed according to some diffuse
probability measure on $\{0,1\}^\infty$.  Consequently, the distribution
of an infinite bridge for the radix sort tree chain with inputs distributed according
to fair coin-tossing measure is a unique mixture of distributions of 
radix sort tree chains with inputs distributed according to diffuse
probability measures on $\{0,1\}^\infty$.  Moreover, an infinite bridge
$(R_n^\infty)_{n \in \bN}$ has an almost surely trivial tail $\sigma$-field if and only if
there is a sequence $(\tt_k)_{k \in \bN}$ with $M(\tt_k) \to \infty$
as $k \to \infty$ such that for all $n \in \bN$ 
the initial segment $(R_1^{\tt_k}, \ldots, R_n^{\tt_k})$ converges in distribution
to $(R_1^\infty, \ldots, R_n^\infty)$ as $k \to \infty$.
\end{theorem}

The structure of the remainder of the paper is as follows.  In Sections \ref{S:forward_probs}, \ref{S:backward_probs},
and \ref{S:Doob-Martin_kernel} we obtain that forward transition probabilities, backward transition probabilities,
and Doob--Martin kernels of the radix sort tree chains.  In Section \ref{S:harmonic_examples} we show that each
radix sort tree chain $(\prescript{\nu}{}{R}_n)_{n \in \bN}$ is a Doob $h$-transform of the Markov chain $(R_n)_{n \in \bN}$.  We consider infinite bridges for the Markov chain $(R_n)_{n \in \bN}$ in Section \ref{S:labeled_infinite_bridges} and introduce an auxiliary consistent labeling of the leaves of the state of the bridge
at each time $n$ by $[n] := \{1,\ldots,n\}$ such that, intuitively, these labelings determine 
a labeling of the limit of the bridge at time $\infty$ and the whole bridge path can be recovered from the limit
and its labeling.  We prove two results, Theorem~\ref{T:main} and Corollarly~\ref{C:main}, 
in Section~\ref{S:T:summary_proof} that together establish Theorem~\ref{T:summary}.

\section{Forward transition probabilities}
\label{S:forward_probs}

Recall that $\bS_n$ is the set of trees that can arise as $\RR(z_1, \ldots, z_n)$
for some choice of distinct $z_1, \ldots, z_n \in \{0,1\}^\infty$.     It is clear that $\RR(z_1, \ldots, z_n)$ is the unique 
finite rooted binary tree $\tt \in \bS_n$ with the following property:  if $\LL(\tt) = \{y_1, \ldots, y_n\}$,
then there is a permutation $\pi$ of $[n]$ such that $z_i \in \tau(y_{\pi(i)})$ for
$i \in [n]$. 

For $n \in \bN$, the distribution of $\prescript{\nu}{}{R}_n$ is specified by
\[
\bP\{\prescript{\nu}{}{R}_1 = \emptyset\} = 1
\]
and, for $n \ge 2$ and $\tt \in \bS_n$ with $\{y_1, \ldots, y_n\} = \LL(\tt)$,
\[
\begin{split}
\bP\{\prescript{\nu}{}{R}_n = \tt\}
& =
\bP\{
\{\zeta_{n,1}(Z_1, \ldots, Z_n), \ldots, \zeta_{n,n}(Z_1, \ldots, Z_n)\} 
=
\{y_1, \ldots, y_n\}\}
\} \\
& = 
n! \prod_{k=1}^n \nu(\tau(y_k)). \\
\end{split}
\] 
In particular,
\begin{align} \label{margindist}
\bP\{{R}_n = \tt\}
=
n! \prod_{k=1}^n \gamma(\tau(y_k)) 
=
n! \prod_{k=1}^n 2^{-|y_k|}.
\end{align}

The  radix sort chain $(\prescript{\nu}{}{R}_n)_{n \in \bN}$ has the following 
forward transition dynamics.
Consider $\ss \in \bS_n$.
There are two classes of trees $\tt \in \bS_{n+1}$ such that 
$$\bP\{\prescript{\nu}{}{R}_{n+1} = \tt \, | \, \prescript{\nu}{}{R}_n = \ss\} > 0.$$

\noindent{\bf Case I.} Here $\tt \in \bS_{n+1}$ is a tree with
$\LL(\tt) = \LL(\ss) \sqcup \{w\}$, where 
$w = x\bar u_m$ for some $x = u_1 u_2 \ldots u_{m-1}$ with
$xu_m \in \ss \setminus \LL(\ss)$.
In this case,
\[
\bP\{\prescript{\nu}{}{R}_{n+1} = \tt \, | \, \prescript{\nu}{}{R}_n = \ss\} = \nu(\tau(w)).
\]
In particular, 
\begin{align}\label{attachtox}
p(\ss,\tt) := \bP\{R_{n+1} = \tt \, | \, R_n = \ss\} = 2^{-|w|} = 2^{-(|x|+1)}.
\end{align}
\noindent{\bf Case II.} Here $\tt \in \bS_{n+1}$ is a tree with
$\LL(\tt) = (\LL(\ss) \setminus \{y\}) \sqcup \{y',y''\}$, where 
$y = u_1 u_2 \ldots u_{m-1} u_m \in \LL(\ss)$,
$y' = u_1 u_2 \ldots u_{m-1} u_m v_1 \ldots v_p$ and 
$y'' = u_1 u_2 \ldots u_{m-1} u_m v_1 \ldots \bar v_p$ for 
some $p \ge 1$ and $v_1, \ldots, v_p \in \{0,1\}$.
In this case,
\[
\begin{split}
\bP\{\prescript{\nu}{}{R}_{n+1} = \tt \, | \, \prescript{\nu}{}{R}_n = \ss\} 
& = \nu(\tau(y')) \frac{\nu(\tau(y''))}{\nu(\tau(y))} + \nu(\tau(y'')) \frac{\nu(\tau(y'))}{\nu(\tau(y))} \\
& = 2 \frac{\nu(\tau(y')) \nu(\tau(y''))}{\nu(\tau(y))}.\\
\end{split}
\]
In particular,
\begin{equation}
\label{attachtoy}
p(\ss,\tt) := \bP\{R_{n+1} = \tt \, | \, R_n = \ss\}
= 2 \frac{2^{-|y'|} 2^{-|y''|}}{2^{-|y|}}.
\end{equation}
For later use we note that, with $\mathbf d:= \mathbf T(v_1 \ldots v_p, v_1 \ldots \bar v_p)$, this may be written as
\begin{align}\label{cherry}
p(\ss,\tt) = 2^{-|y|} \mathbb P\{R_2=\mathbf d\}.
\end{align}

\section{Backward transition probabilities}
\label{S:backward_probs}

Note that if $\ss \in \bS_n$ and $\tt \in \bS_{n+1}$ are such that 
$\bP\{\prescript{\nu}{}{R}_{n+1} = \tt \, | \, \prescript{\nu}{}{R}_n = \ss\} > 0$,
then the leaf set of $\ss$ is obtained either by removing a leaf from the leaf set of $\tt$
that has a sibling which is not a leaf  (corresponding
to Case I above), in which case \eqref{symmetric_tree_build} implies that 
\[
\bP\{\prescript{\nu}{}{R}_n = \ss \, | \, \prescript{\nu}{}{R}_{n+1} = \tt\} = \frac{1}{n+1},
\]
 or by removing two sibling leaves from the leaf set of $\tt$ and replacing them by a single new leaf
positioned at the start of the path that led from the rest of $\tt$ to their common parent
(corresponding to Case II above), in which case \eqref{symmetric_tree_build} implies that 
\[
\bP\{\prescript{\nu}{}{R}_n = \ss \, | \, \prescript{\nu}{}{R}_{n+1} = \tt\} = \frac{2}{n+1}.
\]
These backward transition probabilities can also be obtained directly.  
Again write $\LL(\ss) = \{y_1, \ldots, y_n\}$.  In Case I (using the notation
that was introduced to first describe this case),
\[
\begin{split}
\bP\{\prescript{\nu}{}{R}_n = \ss \, | \, \prescript{\nu}{}{R}_{n+1} = \tt\} 
& =
\frac{
\bP\{\prescript{\nu}{}{R}_n = \ss\}  \bP\{\prescript{\nu}{}{R}_{n+1} = \tt \, | \, \prescript{\nu}{}{R}_n = \ss\}
}
{
\bP\{\prescript{\nu}{}{R}_{n+1} = \tt\}
} \\
& =
\frac{
(n! \prod_{k=1}^n \nu(\tau(y_k))) \nu(\tau(w))
}
{
(n+1)! \prod_{k=1}^n \nu(\tau(y_k))) \nu(\tau(w))
} \\
& =
\frac{1}{n+1}. \\
\end{split}
\]
In Case II (also using the notation
that was introduced to first describe this case),
\[
\begin{split}
\bP\{\prescript{\nu}{}{R}_n = \ss \, | \, \prescript{\nu}{}{R}_{n+1} = \tt\} 
& =
\frac{
\bP\{\prescript{\nu}{}{R}_n = \ss\}  \bP\{\prescript{\nu}{}{R}_{n+1} = \tt \, | \, \prescript{\nu}{}{R}_n = \ss\}
}
{
\bP\{\prescript{\nu}{}{R}_{n+1} = \tt\}
} \\
& =
\frac{
(n! \prod_{k=1}^n \nu(\tau(y_k))) (2 \nu(\tau(y')) \nu(\tau(y'')) / \nu(\tau(y)))
}
{
(n+1)! (\prod_{1 \le k \le n} \nu(\tau(y_k)) / \nu(\tau(y))) \nu(\tau(y')) \nu(\tau(y''))
} \\
& =
\frac{2}{n+1}. \\
\end{split}
\]
The above observations are summarized in the following Definition and Remark.

\begin{definition}
\label{D:kappa}
Suppose that $\tt \in \bS_{n+1}$ and $v = v_1 \ldots v_m$ is a leaf of $\tt$.  If $v_1 \ldots \bar v_m$ 
is not a leaf of $\tt$, let $\kappa(\tt,v) \in \bS_n$ be the tree $\tt \backslash \{v\}$
(that is, $\kappa(\tt, v)$ is the tree with the same leaf set as  $\tt$ except that $v$ has been removed).
If $v_1 \ldots \bar v_m$ is also a leaf of $\tt$, then there is a largest $\ell < m$ such that
$v_1 \ldots v_\ell v_{\ell + 1}$ and $v_1 \ldots v_\ell \bar v_{\ell+1}$ are both  vertices of $\tt$,
and in this case let $\kappa(\tt,v) \in \bS_n$ be the tree 
$\tt \backslash (\{v_1 \ldots v_p : \ell < p \le m\} \cup \{v_1 \ldots \bar v_m\})$
(that is, $\kappa(\tt, v)$ is the tree with the same leaf set as  $\tt$ except that $v$ 
and its sibling leaf $v_1 \ldots \bar v_m$ have both been removed and replaced by 
the single leaf $v_1 \ldots v_\ell$).
\end{definition}  

\begin{remark} 
\label{R:kappa}
Using Definition~\ref{D:kappa}, we can then describe the backward evolution of 
$(\prescript{\nu}{}R_n)_{n \in \bN}$
by saying that conditional on $\{\prescript{\nu}{}R_{n+1}, \prescript{\nu}{}R_{n+2}, \ldots\}$ 
one of the $n+1$ leaves of $\prescript{\nu}{}R_{n+1}$ is chosen
uniformly at random and, denoting this leaf by $V_{n+1}$, the random tree 
$\prescript{\nu}{}R_n$ is constructed as
$\kappa(\prescript{\nu}{}R_{n+1}, V_{n+1})$.
\end{remark}

\section{The Doob-Martin kernel}
\label{S:Doob-Martin_kernel}

Suppose that $\ss \in \bS_m$ and $\tt \in \bS_{m+n}$ are such that $\bP\{R_{m+n} = \tt \, | \, R_m = \ss\} > 0$,
a state of affairs which we denote by $\ss \triangleleft \tt$.  Write $x_1, \ldots, x_p$ for the vertices of
$\ss$ that have degree $2$ and $y_1, \ldots, y_q$ for the leaves of $\ss$.  Of course, $q = m$, but it will
be clearer to use this alternative notation.  Then $\tt$ is obtained from $\ss$ by attaching
subtrees to some of the vertices $\{x_1, \ldots, x_p\} \cup \{y_1, \ldots, y_q\}$.  More precisely, 
$\tt \setminus \ss = (\bigsqcup_{i=1}^p \aa_i) \sqcup  (\bigsqcup_{j=1}^q \bb_j)$
where the subtrees $\aa_i$ and $\bb_j$ are as follows.  
Suppose that $x_i = x_{i1} \ldots x_{i f_i}$ and $u_i \in \{0,1\}$
is such that
$x_{i1} \ldots x_{i f_i} u_i = x_i u_i \notin \ss$, then either 
$\aa_i = \emptyset$ (that is, no subtree is attached to $x_i$, in which case we set $\alpha_i = 0$) or there is an $\alpha_i \ge 1$ and $\cc_i \in \bS_{\alpha_i}$
such that $\aa_i = \{x_i u_i w: w \in \cc_i\}$.
Suppose that $y_j = y_{j1} \ldots y_{j g_j}$, then either $\bb_j = \emptyset$ (that is, no subtree is attached to $y_j$, in which case we set $\beta_j = 0$)
or there is a $\beta_j \ge 1$ and $\dd_j \in \bS_{\beta_j + 1}$ such that
$\bb_j = \{y_j w : w \in \dd_j\} \setminus \{y_j\}$.  We have $n = \sum_i \alpha_i + \sum_j \beta_j$.
Given a tree $\rr \in \bS_h$ for some $h \in \bN$, set $M(\rr) = h$ (so that $M(\rr)$ is the number of leaves of $\rr$) and $\pi(\rr) := \bP\{R_h = \rr\}$.

Then, 
by iterating the arguments that lead to \eqref{attachtox} and \eqref{cherry}, 
\[
\begin{split}
&\bP\{R_{m+n} = \tt \; | \; R_m = \ss\} \\
& \quad =
\frac{n!} {\prod_i \alpha_i! \prod_j \beta_j!} 
\prod_i (2^{-(|x_i|+1)})^{\alpha_i}
\prod_j (2^{-|y_j|})^{\beta_j}
\prod_{\alpha_i \ne 0} \pi(\cc_i)
\prod_{\beta_j \ne 0} \pi(\dd_j). \\
\end{split}
\]
Also, because of \eqref{margindist},
\[
\begin{split}
& \bP\{R_{m+n} = \tt\} =\frac{(m+n)!} {\prod_i \alpha_i! \prod_j \beta_j!} \times\\
& \qquad \qquad \times \prod_{\alpha_i \ne 0} (2^{-(|x_i|+1)})^{\alpha_i} \frac{1}{\alpha_i!}\pi(\cc_i)
\prod_{\beta_j = 0} (2^{-|y_j|})  
\prod_{\beta_j \ne 0} (2^{-|y_j|})^{(\beta_j +1)} \frac{1}{(\beta_j+1)!}\pi(\dd_j) \\
& \quad =
(m+n)!
\prod_i (2^{-(|x_i|+1)})^{\alpha_i}
\prod_j (2^{-|y_j|})^{(\beta_j +1)}
\prod_{\alpha_i \ne 0} \pi(\cc_i)
\prod_{\beta_j \ne 0} \pi(\dd_j).\\
\end{split}
\]
Note also, that
\[
\ss \triangleleft \tt
\Longleftrightarrow
\{v : v \in \LL(\tt), \; y_j \le v\}
\ne \emptyset, \; 1 \le j \le m.
\]
Therefore, the Doob-Martin kernel is
\[
\begin{split}
K(\ss,\tt)
& = \frac{\bP\{R_{m+n} = \tt \; | \; R_m = \ss\}}
{\bP\{R_{m+n} = \tt\}} \\
& =
\frac{\prod_{j=1}^m (\beta_j+1)}{(n+1) \cdots (n+m)} \\
& =
\ind\{\ss \triangleleft \tt\}
\frac{\prod_{j=1}^m 2^{|y_j|} \#\{v : v \in \LL(\tt), \; y_j \le v\}}
{M(\tt) (M(\tt) - 1) \cdots (M(\tt) - m + 1)} \\
& = 
\frac{\prod_{j=1}^m 2^{|y_j|} \#\{v : v \in \LL(\tt), \; y_j \le v\}}
{M(\tt) (M(\tt) - 1) \cdots (M(\tt) - m + 1)}. \\
\end{split}
\]

\begin{remark}
\label{R:nuh_as_limit}
It follows that,  for 
$\ss \in \bS_m$, $m \in \bN$ with leaves $\LL(\ss) = \{y_1, \ldots, y_m\}$
and a sequence $(\tt_n)_{n \in \bN}$ with 
$\lim_{n \to \infty} M(\tt_n) = \infty$, the sequence $K(\ss,\tt_n)$ converges as $n \to \infty$ if and only if 
the limit of
\begin{equation}
\label{Doob-Martin_conv_cond}
\begin{split}
\prod_{j=1}^m
\left[\frac{\#\{v : v \in \LL(\tt_n), \; y_j \le v\}}
{M(\tt_n)}
\bigg / \gamma(\tau(y_j)) \right]\\
\end{split}
\end{equation}
exists, in which case the limits coincide.  Recall that for
$y \in \{0,1\}^\star$ the cardinality
$\#\{1 \le j \le n : y \le \zeta_{n,j}(z_1, \ldots, z_n)\}$ equals
$\#\{1 \le j \le n : y \le z_j\}$ if the latter cardinality is at least two
and it is zero otherwise.
Hence a sufficient condition for the limit as $n \to \infty$ of
$K(\ss,\tt_n)$ (equivalently, of \eqref{Doob-Martin_conv_cond}) to exist 
for all $\ss \in \bS$ is that 
$\tt_n = \RR(z_1, \ldots, z_n)$ for a sequence $(z_n)_{n \in \bN}$
of distinct elements of $\{0,1\}^\infty$
such that for some probability measure $\nu$ on $\{0,1\}^\infty$
we have 
\[
\nu\{z \in \{0,1\}^\infty : y \le z\}
=
\lim_{n \to \infty} \frac{1}{n} \#\{1 \le j \le n: y \le z_j\}
\]
for all $y \in \{0,1\}^\star$; that is, the sequence of empirical probability
distributions $(\frac{1}{n} \sum_{j=1}^n \delta_{z_j})_{n \in \bN}$ converges weakly
to $\nu$ (where we put the usual topology on $\{0,1\}^\infty$ for which
the sets $\tau(y)$ are both closed and open).  In this case 
\begin{equation}
\label{defnuh}
\lim_{n \to \infty} K(\ss,\tt_n)
=
\prescript{\nu}{}h(\ss)
:=
\prod_{a \in \LL(\ss)} \frac{\nu(\tau(a))}{\gamma(\tau(a))}.
\end{equation}
The function $\prescript{\nu}{}h$ is excessive as a pointwise limit of excessive
functions.  Moreover, if $\nu$ is diffuse, then
\[
\lim_{n \to \infty} K(\ss, \prescript{\nu}{}{R}_n) = \prescript{\nu}{}h(\ss), 
\quad \bP-\text{a.s.}
\]
for all $\ss \in \bS$.
\end{remark}

\section{Examples of harmonic functions}
\label{S:harmonic_examples}


It is immediate from the expressions for the forward transition probabilities
derived in Section~\ref{S:forward_probs} that
\[
\bP\{\prescript{\nu}{}{R}_{n+1} = \tt \, | \, \prescript{\nu}{}{R}_n = \ss\}  
=
\prescript{\nu}{}h(\ss)^{-1}
\bP\{{R}_{n+1} = \tt \, | \, {R}_n = \ss\}
\prescript{\nu}{}h(\tt),
\]
where the function $\prescript{\nu}{}h$ was defined in \eqref{defnuh}.

Thus, the nonnegative function $\prescript{\nu}{}h$ is harmonic,
the Markov chain
$(\prescript{\nu}{}{R}_n)_{n \in \bN}$ is the $h$-transform of
$(R_n)_{n \in \bN}$ with the harmonic function $\prescript{\nu}{}h$,
and hence $(\prescript{\nu}{}{R}_n)_{n \in \bN}$ 
is an infinite bridge for $(R_n)_{n \in \bN}$.
Recall that the tail $\sigma$-field of $(\prescript{\nu}{}{R}_n)_{n \in \bN}$
is $\bP$-a.s. trivial.  It follows that the normalized nonnegative
harmonic function $\prescript{\nu}{}h$
is extremal.  We show in Theorem~\ref{T:main} and Corollary~\ref{C:main} that
the extremal normalized nonnegative harmonic functions are precisely those
of this form and that they are, in turn, precisely the harmonic functions that arise as a limit of the form
$\rr \mapsto \lim_{k \to \infty} K(\rr, \tt_k)$, where $(\tt_k)_{k \in \bN}$ is such that
$M(\tt_k) \to \infty$ as $k \to \infty$.  In the language of Doob--Martin theory, this shows that the  
the {\em minimal Doob--Martin boundary} of the radix sort tree chain 
$(R_n)_{n \in \bN}$ coincides with the full {\em Doob--Martin boundary}.  
It may be feasible to prove this fact ``bare--hands'',
but the simpler indirect route we take is, we believe, more informative.

\section{Labeled infinite bridges}
\label{S:labeled_infinite_bridges}

Recall that the backward transition dynamics of
any finite bridge $(R_n^\tt)_{n=1}^{M(\tt)}$ 
and any infinite bridge $(R_n^\infty)_{n \in \bN}$
may be described in terms of the ``pruning'' operation $\kappa$ from
Definition~\ref{D:kappa} and Remark~\ref{R:kappa}:
\begin{itemize}
\item
Suppose that the value of the process at time $n+1$ is $\tt \in \bS_{n+1}$.
\item
Pick a leaf $v$ uniformly at random.
\item
Replace $\tt$ by $\kappa(\tt,v) \in \bS_n$ to produce the value of the process at time $n$.
\end{itemize}

Consider a binary tree $\tt'' \in \bS_{n+1}$.  Label the $n+1$ leaves
of $\tt''$ with $[n+1]$ uniformly at random (that is, all $(n+1)!$
labelings are equally likely).  Let $V$ be the leaf labeled $n+1$.
Set $\tt'  := \kappa(\tt'',V)$.  If the sibling of $V$ was not a leaf
in $\tt''$, then the leaves of $\tt'$ were also leaves of $\tt''$ and
we maintain their labels.  If the sibling of $V$ was also a leaf
of $\tt''$, labeled, say, $k \in [n]$, then in passing from
$\tt''$ to $\tt'$ we remove $V$ and its sibling along with some
vertices on the path leading to their parent, thereby creating
a new leaf which we label $k$ while leaving the labels of
the remaining leaves (which are common to both $\tt''$ and $\tt'$)
unchanged.
The distribution of $\tt'$ is that arising from
one step starting from $\tt''$
of the backward radix sort dynamics (that is, the common backward
dynamics of all infinite bridges). Moreover, the labeling of $\tt'$ 
by $[n]$ is uniformly distributed over the $n!$ possible labelings.

Now suppose that $(R_n^\infty)_{n \in \bN}$ is an infinite bridge.
For some $N$, let $S_N$ be a random binary tree with the same
distribution as $R_N^\infty$.  Label $S_N$ 
uniformly at random with $[N]$ to produce a leaf-labeled binary tree
$\tilde S_N$. The pruning procedure described above is deterministic once
the labeling is given and applying it successively for 
$n=N-1, \ldots, 1$ produces leaf-labeled binary trees 
$\tilde S_{N-1}, \ldots, \tilde S_1$, where $\tilde S_n$ has $n$
leaves labeled by $[n]$ for $1 \le n \le N-1$.  
Write $S_n$ for the underlying binary tree obtained by removing
the labels of $\tilde S_n$.
It follows from the observations above that the sequence $(S_1, \ldots, S_N)$
has the same joint distribution as $(R_1^\infty, \ldots, R_N^\infty)$.
Note that the joint distribution of the sequence
$(\tilde S_1, \ldots, \tilde S_N)$ is uniquely
determined by the distribution of $R_N^\infty$ and hence, 
{\em a fortiori}, by the joint distribution of $(R_n^\infty)_{n \in \bN}$.
Note also that if we perform this construction for two different values
of $N$, say $N' < N''$, to produce, with the obvious notation, sequences
$(\tilde S_1', \ldots, \tilde S_{N'}')$
and
$(\tilde S_1'', \ldots, \tilde S_{N''}'')$,
then
$(\tilde S_1', \ldots, \tilde S_{N'}')$
has the same joint distribution as
$(\tilde S_1'', \ldots, \tilde S_{N'}'')$.

By Kolmogorov's extension theorem we may therefore
suppose that there is a Markov process
$(\tilde R_n^\infty)_{n \in \bN}$
such that for each $n \in \bN$ 
the random element $\tilde R_n$ is a leaf-labeled binary
tree with $n$ leaves labeled by $[n]$
and the following hold.
\begin{itemize}
\item
The binary tree obtained by removing the labels of $\tilde R_n$
is $R_n$.
\item
For every $n \in \bN$, the conditional distribution
of $\tilde R_n^\infty$ given $R_n^\infty$ is uniform over the
$n!$ possible labelings of $R_n^\infty$.
\item
In going backward from time $n+1$ to time $n$,
$\tilde R_{n+1}^\infty$ is transformed into $\tilde R_n^\infty$
according to the deterministic procedure described above.
\end{itemize}

The distribution of the labeled infinite bridge
$(\tilde R_n)_{n \in \bN}$
is uniquely specified by the distribution
of $(R_n^\infty)_{n \in \bN}$
and the above requirements. 
Because of this distributional uniqueness, we refer to 
$(\tilde R_n^\infty)_{n \in \bN}$ as {\bf the}
{\em labeled version} of $(R_n^\infty)_{n \in \bN}$
and $(R_n^\infty)_{n \in \bN}$ as the {\em unlabeled
version} of $(\tilde R_n^\infty)_{n \in \bN}$
and speak of the ``leaf of $R_n^\infty$ labeled with $i \in [n]$ in $\tilde R_n^\infty$.''

\begin{definition}\label{labeling}
Given $i \in [n]$, let $\langle i \rangle_n \in \{0,1\}^\star$ be the leaf of $R_n^\infty$ 
labeled $i$ in $\tilde R_n^\infty$.  Observe that $\langle i \rangle_i \le \langle i \rangle_{i+1} \le \ldots$
and so $\langle i \rangle_\infty = \lim_{n \to \infty} \langle i \rangle_n \in \{0,1\}^\star \sqcup \{0,1\}^\infty$
is well-defined.  Moreover, for distinct $i,j \in \bN$,
$\langle i \rangle_n \wedge \langle j \rangle_n$ is the same for all $n \ge i \wedge j$
and coincides with $\langle i \rangle_\infty \wedge \langle j \rangle_\infty$. 
\end{definition}

\begin{remark}
We have $R_1^\infty \subset R_2^\infty \subset \ldots$ and
\[
R_\infty^\infty := \bigcup_{n \in \bN} R_n^\infty
=
\bigcup_{i \in \bN} 
\{v \in \{0,1\}^\star \sqcup \{0,1\}^\infty : v \le \langle i \rangle_\infty\}.
\]
That is, $R_\infty^\infty$ is the subtree of $\{0,1\}^\star \sqcup \{0,1\}^\infty$
with leaves $\{\langle i \rangle_\infty: i \in \bN\}$ and we define $\tilde R_\infty^\infty$
to be the tree $R_\infty^\infty$ with the leaf $\langle i \rangle_\infty$ labeled $i$, $i \in \bN$.
We will drop the subscripts and write $\langle i \rangle$ for $\langle i \rangle_\infty$, $i \in \bN$.
\end{remark}

\section{Proof of Theorem~\ref{T:summary}}
\label{S:T:summary_proof}

Theorem~\ref{T:summary} is an immediate consequence of Theorem~\ref{T:main} and
Corollary~\ref{C:main} below.

\begin{theorem}
\label{T:main}
Consider an infinite bridge $(R_n^\infty)_{n \in \bN}$ and its associated
labeled version $(\tilde R_n^\infty)_{n \in \bN}$.
\begin{itemize}
\item[(a)]
The sequence $(\langle i \rangle)_{i \in \bN}$ is exchangeable.
\item[(b)]
The tail $\sigma$-field of $(R_n^\infty)_{n \in \bN}$ is $\bP$-a.s.
trivial if and only if $(\langle i \rangle)_{i \in \bN}$ is an independent identically
distributed sequence.
\item[(c)]
If $(\langle i \rangle)_{i \in \bN}$ is  independent and identically
distributed with common distribution $\nu$, then $\nu$ is concentrated on
$\{0,1\}^\infty$ and diffuse.
\item[(d)]
The tail $\sigma$-field of $(R_n^\infty)_{n \in \bN}$ is $\bP$-a.s.
trivial if and only if $(R_n^\infty)_{n \in \bN}$ has the same distribution
as $(\prescript{\nu}{}R_n)_{n \in \bN}$ for some diffuse probability measure
$\nu$ on $\{0,1\}^\infty$.
\end{itemize}
\end{theorem}

\begin{proof}
(a) It is clear by construction that $(\langle i \rangle_n)_{i \in [n]}$ is 
(finitely) exchangeable and the claim follows upon taking limits as $n \rightarrow \infty$.

\noindent
(b) The bijective correspondence between the distributions of the infinite bridges $(R_n^\infty)_{n \in \bN}$ 
and the distributions of their labeled versions $(\tilde R_n^\infty)_{n \in \bN}$ 
is compatible with convex combinations, and hence preserves extremality. 
Therefore the  tail $\sigma$-field of the infinite bridge $(R_n^\infty)_{n \in \bN}$ is $\bP$-a.s.
trivial if and only if the exchangeable sequence $(\langle i \rangle)_{i \in \bN}$ is ergodic.
(This situation closely parallels one appearing in the analysis of R\'emy's tree growth chain in \cite{Remy}, 
and we refer to the more detailed argument in Proposition 5.19 (see also the subsequent Remark 5.20) of \cite{Remy}.)
Finally,  a well-known consequence of de Finetti's theorem is that an exchangeable sequence is
ergodic if and only if it is independent and identically distributed.

\noindent
(c) For any $u \in \{0,1\}^\star$, the sequence $(\ind\{u = \langle k \rangle\})_{k \in \bN}$ is
independent and identically distributed, and hence
$\#\{k \in \bN: u = \langle k \rangle\} = 0$ $\bP$-a.s.
or
$\#\{k \in \bN: u = \langle k \rangle\} = \infty$ $\bP$-a.s.
Now, if $\bP\{\langle i \rangle \in \{0,1\}^\star\} > 0$ there would be a $u \in \{0,1\}^\star$
such that with positive probability $\langle i \rangle_n = \langle i \rangle = u$ for all $n$
sufficiently large. Then, on the event $\{\langle i \rangle = u\}$ we would have $\#\{k \in \bN:  \langle k \rangle =u\} = 1$, 
 since it follows from the construction in Definition \ref{labeling}
that $\langle j \rangle \ne \langle i \rangle$ for $j \ne i$ when $\langle i \rangle \in \{0,1\}^\star$. This shows that $\bP\{\langle i \rangle \in \{0,1\}^\star\} = 0$.

We therefore have that $(\langle k \rangle)_{k \in \bN}$ is an independent identically distributed sequence of
$\{0,1\}^\infty$-valued random variables.  Because 
$\langle i \rangle \wedge  \langle j \rangle = \langle i \rangle_n \wedge  \langle j \rangle_n \in \{0,1\}^\star$
for all $n \ge i \vee j$ $\bP$-a.s. when $i \ne j$, it follows that
$\langle i \rangle \ne \langle j \rangle$ $\bP$-a.s. for $i \ne j$ and the common distribution of
$(\langle k \rangle)_{k \in \bN}$ is diffuse. 

\noindent
(d)
We have already seen that when $\nu$ is a diffuse probability measure
on $\{0,1\}^\infty$ the process $(\prescript{\nu}{}R_n)_{n \in \bN}$ 
is an infinite bridge which, by the Hewitt-Savage zero-one law, has a trivial tail $\sigma$-field.

Conversely, suppose that the infinite bridge $(R_n^\infty)_{n \in \bN}$ has a trivial tail $\sigma$-field.
Let $\nu$ be the common diffuse distribution of the independent, identically distributed sequence
of $\{0,1\}^\infty$-valued random variables $(\langle i \rangle)_{i \in \bN}$.  
In the notation of the Introduction, it is clear that 
$R_n^\infty = \RR(\langle 1 \rangle, \ldots, \langle n \rangle)$, $n \in \bN$, and so
$(R_n^\infty)_{n \in \bN}$ has the same distribution as $(\prescript{\nu}{}R_n)_{n \in \bN}$.
\end{proof}

\begin{corollary}
\label{C:main}
The extremal normalized nonnegative harmonic functions are precisely those that arise as
$\ss \mapsto \lim_{k \to \infty} K(\ss, \tt_k)$ for a sequence $(\tt_k)_{k \in \bN}$ with 
$M(\tt_k) \to \infty$ as $k \to \infty$. 
There is a bijective correspondence between diffuse probability measures on $\{0,1\}^\infty$
and such functions:  the measure $\nu$ corresponds to the normalized nonnegative harmonic function
$\prescript{\nu}{}h$ of \eqref{defnuh} and, conversely, if $h$ is
an extremal normalized nonnegative harmonic function and 
$(R_n^\infty)_{n \in \bN}$ is the infinite bridge constructed as the Doob $h$-transform
of $(R_n)_{n \in \bN}$ using the function $h$, then $h=\prescript{\nu}{}h$,
where $\nu$ is the common
distribution of the independent identically distributed sequence $(\langle i \rangle)_{i \in \bN}$
associated with the labeled infinite bridge $(\tilde R_n^\infty)_{n \in \bN}$.
\end{corollary}

\begin{proof}
We know from Theorem~\ref{T:main} that the extremal normalized nonnegative harmonic functions 
correspond to infinite bridges of the form $(\prescript{\nu}{}R_n)_{n \in \bN}$ where $\nu$ is 
a diffuse probability measure on $\{0,1\}^\infty$,
and hence they are the harmonic functions~$\prescript{\nu}{}h$.  
In order to see that the correspondence between $\nu$ and the distribution of  $(\prescript{\nu}{}R_n)_{n \in \bN}$ is bijective,
we observe that $\nu$ is determined uniquely by 
the distribution of the labeled version of $(\prescript{\nu}{}R_n)_{n \in \bN}$ 
and hence by the distribution of $(\prescript{\nu}{}R_n)_{n \in \bN}$ itself.

It remains to check that if the normalized nonnegative harmonic function $h$ is given by
$h(\ss) = \lim_{k \to \infty} K(\ss, \tt_k)$ for a sequence $(\tt_k)_{k \in \bN}$ with 
$M(\tt_k) \to \infty$ as $k \to \infty$, then $h$ is extremal.  We will follow
an argument similar to the proof of Corollary 5.21 in \cite{Remy}.  
Writing $(R_n^\infty)_{n \in \bN}$ for the
infinite bridge given by the Doob $h$-transform of $(R_n)_{n \in \bN}$ associated with $h$,
we recall that extremality of $h$ is equivalent to the tail
$\sigma$-field of $(R_n^\infty)_{n \in \bN}$ being $\bP$-a.s. trivial.  
By Theorem~\ref{T:main}, this is in turn equivalent to showing that the exchangeable sequence 
$(\langle i \rangle)_{i \in \bN}$ has the equivalent properties of being
ergodic or independent and identically distributed.

Note that $\langle i \rangle$ is the unique $v \in \{0,1\}^\infty$ such that
$\langle i \rangle \wedge \langle j \rangle \le v$ for all $j \ne i$.
It follows that there is a measurable bijection mapping the sequence
$(\langle i \rangle)_{i \in \bN}$ to the jointly exchangeable $\{0,1\}^\star$-valued array 
$\{\langle i \rangle \wedge \langle j \rangle: i,j \in \bN, \; i \ne j\}$ in such a way
that the sequence will be ergodic if and only if the array is ergodic.  
By a result of Aldous
(see, for example, \cite[Lemma~7.35]{MR2161313}),
the array is ergodic if and only if for any disjoint finite subsets $H_1, \ldots, H_s$
of $\bN$ the finite subarrays 
$\{\langle i \rangle \wedge \langle j \rangle: i,j \in H_r, \; i \ne j\}$, $1 \le r \le s$,
are independent.

Recall that
$(R_1^{\tt_k}, \ldots, R_{M(\tt_k)}^{\tt_k})$ denotes the bridge to $\tt_k$.
For any $\ell \in \bN$, $R_\ell^{\tt_k}$ converges in distribution to $R_\ell^\infty$ as $k \to \infty$.
We can build a labeled version $(\tilde R_1^{\tt_k}, \ldots, \tilde R_{m(\tt_k)}^{\tt_k})$
of $(R_1^{\tt_k}, \ldots, R_{m(\tt_k)}^{\tt_k})$ in much the same way that we built a labeled version
of an infinite bridge: $\tilde R_{m(\tt_k)}^{\tt_k}$ consists of 
the tree $R_{m(\tt_k)}^{\tt_k} = \tt_k$ with its $M(\tt_k)$ leaves labeled uniformly at random with the set
$[M(\tt_k)]$ and the backward evolution of such a labeled finite bridge is the same 
as that of the labeled infinite bridge.  It is clear that 
$\tilde R_\ell^{\tt_k}$ converges in distribution to $\tilde R_\ell^\infty$ as $k \to \infty$
for all $\ell \in \bN$: indeed,  $\tilde R_\ell^{\tt_k}$ and $\tilde R_\ell^\infty$
are just $R_\ell^{\tt_k}$ and $R_\ell^\infty$, respectively, equipped with uniform random labelings
of their $\ell$ leaves by the set $[\ell]$.  

Write $\langle i \rangle_\ell^k$ for the element of 
$\{0,1\}^\star$ labeled $i$ in  $\tilde R_\ell^{\tt_k}$ for $1 \le i \le \ell \le M(\tt_k)$.
The finite array $\{\langle i \rangle_\ell^k \wedge \langle j \rangle_\ell^k : 1 \le i \ne j \le \ell\}$
converges in distribution to the finite array 
$\{\langle i \rangle_\ell \wedge \langle j \rangle_\ell : 1 \le i \ne j \le \ell\} 
=
\{\langle i \rangle \wedge \langle j \rangle : 1 \le i \ne j \le \ell\}$ as $k \to \infty$.

Write $u_1^k, \ldots, u_{M(\tt_k)}^k$ for the leaves of $\tt_k$.  Suppose that $I_1^k, \ldots, I_{M(\tt_k)}^k$
is a listing of $[M(\tt_k)]$ in uniform random order and $J_1^k, \ldots, J^k_{M(\tt_k)}$ is a sequence
of independent random variables uniformly distributed on $[M(\tt_k)]$. By definition,
$(\langle i \rangle_\ell^k)_{1 \le i \le \ell}$ has the same distribution as
$(u_{I_i^k}^k)_{1 \le i \le \ell}$.  We may couple $I_1^k, \ldots, I_{M(\tt_k)}^k$ and $J_1^k, \ldots, J_{M(\tt_k)^k}$
together on the same probability space in such a way that
$\lim_{k \to \infty} \bP\{\exists 1 \le i \le \ell : I_i^k  \ne J_i^k\} = 0$ and hence
$\lim_{k \to \infty} \bP\{\exists 1 \le i \ne j \le \ell : u_{I_i}^k \wedge u_{I_j^k}^k \ne u_{J_i^k}^k \wedge u_{J_j^k}^k\} = 0$.
If $H_1, \ldots, H_s$ is a collection of disjoint subsets of $[\ell]$, and $k$ is so large that $M(\tt_k \ge \ell$, then it is clear that the arrays
$\{u_{J_i^k}^k \wedge u_{J_j^k}^k : i,j \in H_r, \; i \ne j\}$, $1 \le r \le s$, are independent and hence the arrays
$\{\langle i \rangle \wedge \langle j \rangle: i,j \in H_r, \; i \ne j\}$, $1 \le r \le s$, are also independent,
as required.
\end{proof}

\section{Examples of excessive functions}
\label{S:excessive_examples}

We saw in Section~\ref{S:harmonic_examples} that for a diffuse probability measure 
$\nu$ the excessive function
$\prescript{\nu}{}h$ of \eqref{defnuh} is actually harmonic.
The definition of  
$\prescript{\nu}{}h$ still makes sense when $\nu$ is not diffuse
and it is interesting to investigate the properties of this excessive function in that case.

Let $G$ be the potential kernel (that is, the Green kernel) for $(R_n)_{n \in \bN}$, which in our situation is given by $G(\ss, \tt) = \mathbb P\{R_k = \tt | R_m = \ss\}$ for $\ss \in \mathbb S_m$, $\tt \in \mathbb S_k$.
Because the function $\prescript{\nu}{}h$ is excessive, we have the Riesz decomposition 
$\prescript{\nu}{}h(\rr) = H(\rr) + \sum_{\ss \in \bS} G(\rr,\ss) \, \eta(\ss)$
for some nonnegative harmonic function $H$ and measure $\eta$ determined by
\[
\eta(\ss) 
= 
\prescript{\nu}{}h(\ss)
-
\sum_\tt 
p(\ss,\tt)\, 
\prescript{\nu}{}h(\tt).
\]
We claim that $H \equiv 0$ so that $\prescript{\nu}{}h$ is a pure potential.

Using the notation of Section~\ref{S:forward_probs} with the first sum for Case I and the second sum for Case II,
\[
\begin{split}
& \sum_\tt p(\ss,\tt) \prescript{\nu}{}h(\tt) \\
& \quad =
\sum_{w} 2^{-|w|} \, \prescript{\nu}{}h(\ss)  2^{|w|} \nu(\tau(w)) \\
& \qquad +
\sum_{y,y',y''} 2 \frac{2^{-|y'|} 2^{-|y''|}}{2^{-|y|}} \, 
\prescript{\nu}{}h(\ss) 
\frac{2^{|y'|} \nu(\tau(y')) 2^{|y''|} \nu(\tau(y''))}
{2^{|y|} \nu(\tau(y))} \\
& \quad =
\prescript{\nu}{}h(\ss)
\left[
\sum_{w} \nu(\tau(w))
+
\sum_{y \in \LL(\ss)}
\nu(\tau(y)) 
\sum_{y',y''}
2 \frac{\nu(\tau(y'))}{\nu(\tau(y))} \frac{\nu(\tau(y''))}{\nu(\tau(y))}
\right], \\
\end{split}
\]
where the summation in the first sum of the middle and right
members is over  
$w = u_1 u_2 \ldots u_{m-1} \bar u_m \notin \ss$
such that $u_1 u_2 \ldots u_{m-1} u_m \in \ss \setminus \LL(\ss)$ (Case I),
and the summation in the second sum of these members is over
$y = u_1 u_2 \ldots u_{m-1} u_m \in \LL(\ss)$,
$y' = u_1 u_2 \ldots u_{m-1} u_m v_1 \ldots v_p \notin \ss$, and 
$y'' = u_1 u_2 \ldots u_{m-1} u_m v_1 \ldots \bar v_p \notin \ss$ for 
some $p \ge 1$ and $v_1, \ldots, v_p \in \{0,1\}$ (Case II).

Now 
\[
\sum_{w} \nu(\tau(w)) = 1 - \sum_{y \in \LL(\ss)} \nu(\tau(y))
\]
where again the range of summation for $w$ is as in Case I,
and for $y\in \LL(\ss)$
\[
2 \sum_{y',y''}  
\nu(\tau(y')) \nu(\tau(y''))
=
\nu \otimes \nu\{(x', x'') : y < x', \, y < x'', \, x' \ne x''\}
\]
where again the range of summation for $y',y''$ is as in Case II.
Therefore,
\[
\begin{split}
& \prescript{\nu}{}h(\ss)
-
\sum_\tt p(\ss,\tt)\, \prescript{\nu}{}h(\tt) \\
& \quad =
\prescript{\nu}{}h(\ss)
\sum_{y \in \LL(\ss)} \nu(\tau(y))
\bar \nu_y \otimes \bar \nu_y\{(x', x'') : x' = x''\}, \\
\end{split}
\]
where we write $\bar \nu_y$ for the restriction of $\nu$ to $\tau(y)$ normalized to be
a probability measure (if $\nu(\tau(y)) = 0$ we define $\bar \nu_y$ arbitrarily).
Thus, the measure appearing in the Riesz decomposition 
of the excessive function $\prescript{\nu}{}h$ is given by
\[
\eta(\ss) = \prod_{a \in \LL(\ss)} 2^{|a|} \nu(\tau(a))
\sum_{y \in \LL(\ss)} \nu(\tau(y))
\bar \nu_y \otimes \bar \nu_y\{(x', x'') : x' = x''\}.
\]

By general theory, $\prescript{\nu}{}h$ has the Choquet representation
\[
\prescript{\nu}{}h(\rr)
=
\int_{\partial \bS} K(\rr,b) \, \theta(db) + \sum_{\ss \in \bS} K(\rr,\ss) \, \theta(\ss),
\]
where $\partial \bS$ is the Doob--Martin boundary, $K(\rr,b)$, $\rr \in \bS$, $b \in \partial \bS$
is the extended Doob--Martin kernel, and $\theta$ is a probability measure on 
$\bar \bS := \partial \bS \cup \bS$.

Recalling from \eqref{margindist} that $G(\emptyset, \ss) = \# \LL(\ss) ! \prod_{a \in \LL(\ss)} 2^{-|a|}$, we have
\[
\theta(\ss)
= M(\ss) ! 
\prod_{a \in \LL(\ss)} \nu(\tau(a))
\sum_{y \in \LL(\ss)} \nu(\tau(y))
\bar \nu_y \otimes \bar \nu_y\{(x', x'') : x' = x''\}.
\]
Letting $Z_1, Z_2, \ldots$ be i.i.d. $\{0,1\}^\infty$-valued random variables with common distribution
$\nu$ we can write, with $n = M(\ss)$,
\[
\begin{split}
\theta(\ss) 
&=
\bP\biggl(\{Z_i \ne Z_j, \, 1 \le i \ne j \le n, \; Z_{n+1} = Z_k  \; \text{for some $1 \le k \le n$}\} \\
& \qquad \cap \{\RR(Z_1, \ldots, Z_n) = \ss \}
\biggr). \\
\end{split}
\]
Thus,
$\sum_{s \in \bS} \theta(s) = \bP\{\exists 1 \le i < j < \infty : Z_i = Z_j\} = 1$ whenever 
$\nu$ has a nontrivial discrete component, and so the function $\prescript{\nu}{}h$ is indeed
a pure potential in this case.  

By arguments similar to those in Section~\ref{S:harmonic_examples},
it is possible to check that the Doob $h$-transform of $(R_n)_{n \in \bN}$
built from the excessive function $\prescript{\nu}{}h$ can be constructed as follows:
let $Z_1, Z_2, \ldots$ be i.i.d. with common distribution $\nu$
and while $Z_1, \ldots, Z_n$ are distinct the value of the chain
is $\RR(Z_1, \ldots, Z_n)$, but the chain is killed and sent to the
cemetery at the first time $n$ such that $Z_n$ is equal 
to one of the previously observed values
$\{Z_1, \ldots, Z_{n-1}\}$.  We denote this killed Markov chain
by $(\prescript{\nu}{}R_n)_{n \in \bN}$, just as we did when $\nu$ is diffuse.

In general, for each $\ss \in \bS$ the function $\nu \mapsto \prescript{\nu}{}h(\ss)$
is continuous with respect to the topology of weak convergence of probability
measures on $\{0,1\}^\infty$.  Similarly, the mapping from $\nu$ to the
distribution of $(\prescript{\nu}{}R_n)_{n \in \bN}$ is continuous provided
that we identify the cemetery state with the point at infinity in the one-point
compactification of $\bS$.

We note that unlike the situation when $\nu$ is diffuse, different choices of
$\nu$ with a discrete component
can result in the same distribution for 
$(\prescript{{\nu}}{}R_n)_{n \in \bN}$.  
For example, write 
$a = (0,0,\ldots)$,  
$b = (0,1,1,\ldots)$, 
$c = (1,1,\ldots)$, 
and 
$d = (1,0,0,\ldots)$,
and put 
$\nu_1=\frac{1}{3}\delta_{a} + \frac{2}{3} \delta_{c}$,
$\nu_2=\frac{1}{3}\delta_{b} + \frac{2}{3} \delta_{d}$,
$\nu_3=\frac{2}{3}\delta_{a} + \frac{1}{3} \delta_{c}$,
and
$\nu_4=\frac{2}{3}\delta_{b} + \frac{1}{3} \delta_{d}$.
Denote the cemetery state by $\dag$ and let $\tt$ be the tree with the
three vertices $\emptyset, 0, 1$.  Then, for $1 \le j \le 4$,
\[
\bP\{\prescript{{\nu_j}}{}R_1 = \emptyset, \, \prescript{{\nu_j}}{}R_2 = \tt, \, \prescript{{\nu_j}}{}R_3 = \dag\}
= \frac{4}{9}
\]
and
\[
\bP\{\prescript{{\nu_j}}{}R_1 = \emptyset, \, \prescript{{\nu_j}}{}R_2 = \dag\}
= \frac{5}{9},
\]
so that the chains $(\prescript{{\nu_j}}{}R_n)_{n \in \bN}$, $1 \le j \le 4$, 
have the same distribution.
Observe that $\prescript{{\nu_j}}{}h$ is the same for each $j$, whereas
when $\nu^*$ and $\nu^{**}$ are different diffuse probability distributions
the fact that the distributions  of $(\prescript{\nu^*}{}R_n)_{n \in \bN}$ 
and $(\prescript{\nu^{**}}{}R_n)_{n \in \bN}$ differ certainly implies
that $\prescript{\nu^*}{}h \ne \prescript{\nu^{**}}{}h$. 

\bigskip\noindent
{\bf Acknowledgments:}  We thank Kevin Leckey and Ralph Neininger for valuable information about the literature
around radix sort algorithms.

\def\cprime{$'$}
\providecommand{\bysame}{\leavevmode\hbox to3em{\hrulefill}\thinspace}
\providecommand{\MR}{\relax\ifhmode\unskip\space\fi MR }
\providecommand{\MRhref}[2]{%
  \href{http://www.ams.org/mathscinet-getitem?mr=#1}{#2}
}
\providecommand{\href}[2]{#2}

\end{document}